\documentclass[11pt,dvips,twoside,letterpaper]{article}
\usepackage{pslatex}
\usepackage{fancyhdr}
\usepackage{graphicx}
\usepackage{geometry}
 \RequirePackage[T1]{fontenc}

\def\figurename{Figure} % Replace the colon that normally appears after the Figure number by a period.
\makeatletter
\renewcommand{\fnum@figure}[1]{\figurename~\thefigure.}
\makeatother

\def\tablename{Table} % Replace the colon that normally appears after the Figure number by a period.
\makeatletter
\renewcommand{\fnum@table}[1]{\tablename~\thetable.}
\makeatother

\usepackage{amsmath}
\usepackage{amssymb}
\usepackage{amsfonts}
\usepackage{amsthm,amscd}

\newtheorem{theorem}{Theorem}[section]
\newtheorem{lemma}[theorem]{Lemma}
\newtheorem{corollary}[theorem]{Corollary}
\newtheorem{proposition}[theorem]{Proposition}
\theoremstyle{definition}
\newtheorem{definition}[theorem]{Definition}

\theoremstyle{remark}
\newtheorem{remark}[theorem]{Remark}

\numberwithin{equation}{section}

\def\P{\mathbb P}

\def\R{\mathbb R}
\def\E{\mathbb E}

\def\E{\mathbb E}

\begin{document}
%\vskip 0.4in
\title{ Multivalued stochastic partial differential-integral equations via
backward doubly stochastic differential equations driven by a
L\'{e}vy process \thanks{The first author is supported by the
National Natural Science Foundation of
 China (No. 10901003) and the Great Research Project of Natural
 Science Foundation of Anhui Provincial Universities (No. KJ2010ZD02).
 The second author is
supported by TWAS Research Grants to individuals (No. 09-100
RG/MATHS/AF/\mbox{AC-I}--UNESCO FR: 3240230311)}}
\author{Yon Ren$^1$\thanks{brightry@hotmail.com and
renyong@126.com}\; and Auguste Aman$^2$\thanks{augusteaman5@yahoo.fr}   \\
1.\;{\it Department of Mathematics, Anhui Normal University, Wuhu
241000, China}\;\;\;\;\;\;\;\;\;\;\;\;\;\;\;\;\;\;\;\;\;\\
2.\;{\it U.F.R Mathématiques et informatique, Universit\'{e} de Cocody,\,582 Abidjan 22, C\^{o}te d'Ivoire}
}
\date{}
\maketitle

\begin{abstract}
In this paper, we deal with a class of backward doubly stochastic
differential equations (BDSDEs, in short) involving subdifferential
operator of a convex function and driven by Teugels martingales
associated with a L\'{e}vy process. We show the existence and
uniqueness result by means of Yosida approximation. As an
application, we give the existence of stochastic viscosity solution
for a class of multivalued stochastic partial differential-integral
equations (MSPIDEs, in short).
\end{abstract}

\textbf{Key Words}: {Backward doubly
stochastic differential
 equation\and subdifferential operator,\, L\'{e}vy process,\; Teugels martingale,\;
 multivalued stochastic partial differential-integral equation.

\section{Introduction}
\label{intro}
Backward stochastic differential equations (BSDEs, in short) related
to a multivalued maximal monotone operator defined by the
subdifferential of a convex function have first been introduced by
Gegout-Petit and Pardoux \cite{Geg}. Further, Pardoux and
R\u{a}\c{s}canu \cite{PR} proved the existence and uniqueness of the
solution of BSDEs, on a random (possibly infinite) time interval,
involving a subdifferential operator in order to give a
probabilistic interpretation for the viscosity solution of some
parabolic and elliptic variational inequalities. Following, Ouknine
\cite{Ou}, N'zi and Ouknine \cite{Mn1,Mn2}, Bahlali et al.
\cite{Ba1,Ba2} discussed this type of BSDEs driven by a Brownian motion
or the combination of a Brownian motion and an independent Poisson
point process under the conditions of Lipschitz, locally Lipschitz
or some monotone conditions on the coefficients.

 Recently, a
new class of BSDEs, named backward doubly stochastic differential
equations (BDSDEs, in short) involving a standard forward stochastic
integral and a backward stochastic integral has been introduced by
Pardoux and Peng \cite{PP} in order to give a probabilistic
representation for a class of quasilinear stochastic partial
differential equations (SPDEs, in short). Following it,  Matoussi
and Scheutzow \cite{Mou}, Bally and Matoussi \cite{Bally}, Zhang and
Zhao \cite{Zhang}, Aman and Mrhardy\cite{AM} and Boufoussi et al. \cite{B1,B2} studied this kind
of BDSDEs from different aspects.

The main tool in the theory of BSDEs is the martingale
representation theorem for a martingale which is adapted to the
filtration of a Brownian motion or a Poisson point process (Pardoux
and Peng \cite{PP1}, Tang and Li \cite{Ta}). Recently, Nualart and
Schoutens \cite{Nu1} gave a martingale representation theorem
associated with a L\'{e}vy process. This class of L\'{e}vy processes
includes Brownian motion, Poisson process, Gamma process, negative
binomial process and Meixner process as special cases. Based on
\cite{Nu1}, they showed the existence and uniqueness of the solution
for BSDEs driven by Teugels martingales associated with a L\'{e}vy
process in \cite{Nu2}. These results were important from a pure
mathematical point of view as well as from application point of view
in the world of finance. Specifically, they could be used for the
purpose of option pricing in a L\'{e}vy market and related partial
differential equation which provided an analogue of the famous
Black-Scholes formula. Motivated by \cite{PP} and \cite{Nu2}, Ren et
al. \cite{Ren2} considered a class of BDSDEs driven by Teugels
martingales and an independent Brownian motion, obtained the
existence and uniqueness of solutions to these equations, which
allowed to give a probabilistic interpretation for the solution to a
class of stochastic partial differential-integral equations (SPDIEs,
in short). Very recently, Ren and Fan \cite{ren3} derived the
existence and uniqueness of the solution for BSDEs driven by a
L\'{e}vy process involving a subdifferential operator and gave a
probabilistic interpretation for the solutions of a class of partial
differential-integral inclusions (PDIIs, in short).

Motivated by the above works, the first aim of this paper is to
derive existence and uniqueness result to the following BDSDE
involving subdifferential operator of a convex function and driven by
Teugels martingales associated with a L\'{e}vy process
\begin{eqnarray}\label{equation1}
\left\{
\begin{array}{ll}
\,{\rm d}Y_t+f(t,Y_t,Z_t)\,{\rm d}t+g(t,Y_t,Z_t)\,{\rm d}B_t\in
\partial \varphi(Y_t)\,{\rm d}t+\sum_{i=1}^{\infty}Z^{(i)}_t\,{\rm
d}H^{(i)}_t
,\ 0\leq t\leq T, \\\\
Y_T=\xi,
\end{array}\right.
\label{i1}
\end{eqnarray}
where $\partial \varphi$ is a subdifferential operators. The
integral with respect to $\{B_t\}$ is a backward Kunita-It\^{o}
integral (see Kunita \cite{Kunita}) and this one with respect to
$\{H^{(i)}_t\}_{i\geq 1}$ is a standard forward It\^{o} integral
(see Gong \cite{Gong}). Our method is based on the Yosida
approximation.

On the other hand, since the pioneering paper due to Buckdahn and Ma
\cite{BMa},\cite{BM1},\cite{BM2}, the notion of stochastic viscosity solution has
been intensely studied in the last ten year. Among others, we can
cite the work of Boufoussi et al. \cite{B1}, \cite{B2}, Aman and Mrhardy
\cite{AM}, Aman and Ren \cite{MY} and Ren et al. \cite{Yal}, etc. In
all these different works, authors have set existence results to
stochastic viscosity solution  of several types of SPDE. The tool is
entirely probabilistic and used the connection between these SPDE
and associated BDSDEs. Following this way, the second goal in this
paper is to give stochastic viscosity solution for multivalued
stochastic partial differential-integral equations (MSPDIEs, in
short)
\begin{eqnarray}
\left\{
\begin{array}{ll}
\left(\frac{\partial u}{\partial
t}(t,x)+\mathcal{L}u(t,x)+f\left(t,x,u(t,x),(u_k^1(t,x))_{k=1}^{\infty}\right)
+g(t,x,u(t,x))\dot{B}_t\right)\in\partial\varphi(x),\;\; 0<t<T,\;
x\in{\R}^{d}
\\\\
u(T,x)=u_0(x),\;\;x\in\R^{d},
\end{array}\right.
\label{i11}
\end{eqnarray}
where $\mathcal{L}$ is the second-order differential integral
operator of the diffusion process given by
\begin{eqnarray}
\mathcal{L}\phi(t,x)&=&m_1\sum_{i=1}^{d}\sigma_{i}(x)\frac{\partial\phi}
{\partial x_{i}}(t,x)+\frac{1}{2}\sum_{i=1}^{d}\sigma^2_{i}(x)
\frac{\partial^2\phi}{\partial x_i^2}(t,x)\nonumber\\&&+\int_{\R}
\left[\phi(t,x+\sigma(x)y)-\varphi(t,x)
-\langle\nabla\phi(t,x),\sigma(x)y\rangle\right]\nu({\rm d}y),\nonumber\\
\label{i2}
\end{eqnarray}
and
\begin{eqnarray*}
\phi^{1}_k(t,x)=\int_{\R}(\phi(t,x+\sigma(x)y)-\phi(t,x))p_k(y)\nu({\rm
d}y),
\end{eqnarray*}
with $\sigma$ a $\R^d$-valued function, coefficient of SDE driven by
the Lévy process $L$, $m_1=\E(L_1)$ and $p_k$ precise later  . Our method is also fully
probabilistic and uses connection between MSPDIE (\ref{i1}) and
BDSDE (\ref{equation1}) in Markovian framework.

The paper is organized as follows. In Section 2, we introduce some
preliminaries and notations. Section 3 is devoted to the existence and uniqueness result for BDSDEs involving
subdifferential operator of a convex function and driven by a L\'{e}vy
process. Finally, in section 4 we derive a probabilistic
representation (in stochastic viscosity sense) for the solution of a
class of MSPDIEs via BDSDEs proposed in Section 3.

\section{Preliminaries and notations}

Let $T>0$ be a fixed terminal time and $\{B_t:t\in [0,T]\}$ be a
standard $\mathbb{R}$-valued Brownian motion defined on a complete
probability space $(\Omega, \mathcal{F},\mathbb{P})$. Let us also
consider $\{L_t:t\in [0,T]\}$, a $\mathbb{R}$-valued L\'{e}v process
corresponding to a standard L\'{e}vy measure $\nu$, defined on a
complete probability space $(\Omega',\mathcal{F}',\mathbb{P}')$, with the following
characteristic function:
\begin{eqnarray*}
\mathbb{E}({\rm
e}^{iuL_t})=\exp\left[iaut-\frac{1}{2}\kappa^2u^2t+t\int_{\mathbb{R}}({\rm
e}^{iux}-1-iux1_{\{|x|<1\}})\nu({\rm d}x)\right],
\end{eqnarray*}
where $a\in \mathbb{R}, \kappa\geq 0$. Moreover, the Lévy measure
$\nu$ satisfies the following conditions:
\begin{enumerate}
\item $\int_{\mathbb{R}}(1\wedge y^2)\nu({\rm d}y)< \infty;$
\item $\int_{]-\varepsilon,\varepsilon[^c}{\rm e}^{\lambda |y|}\nu({\rm d}y)< \infty,$ for every $\varepsilon>0$
  and for some $\lambda>0$;
\end{enumerate}
which provides that $L_t$ has moments of all orders, i.e.
$\int_{-\infty}^{+\infty}|x|^i\nu({\rm d}x)<\infty,\ \forall i\geq
2$.

Let consider the product space $(\bar{\Omega},\bar{\mathcal{F}}, \bar{\P})$, defined by
\begin{eqnarray*}
\bar{\Omega}=\Omega\times\Omega';\;\;\;\;\;
\bar{\mathcal{F}}=\mathcal{F}\otimes\mathcal{F}'\;\;;\;\;\;\;\;
\bar{\P}=\P\otimes\P'.
\end{eqnarray*}
Further, random variables
$\xi(\omega),\;\omega\in \Omega$ and $\zeta(\omega'),\; \omega'\in
\Omega'$ can be considered as random variables on $\bar{\Omega} $ via
the following identifications:
\begin{eqnarray*}
\xi(\omega,\omega')=\xi(\omega);\,\,\,\,\,\zeta(\omega,\omega')=\zeta(\omega').
\end{eqnarray*}
In this fact, the processes $B$ and $L$ are assumed independent.
Next, denoting by $\mathcal{N}$ the totality of $\bar{\mathbb{P}}$-null sets of
$\bar{\mathcal{F}}$, and for each $t\in [0,T]$, we define
\begin{eqnarray*}
\mathcal{F}_t=\mathcal {F}_{t,T}^B\otimes \mathcal {F}_{t}^L\vee
\mathcal{N},
\end{eqnarray*}
where for any process $\{\eta_t\}, \mathcal{F}_{s,t}^\eta=\sigma\{\eta_r-\eta_s:s\leq r\leq t\}$ and $\mathcal {F}_{t}^\eta=\mathcal {F}_{0,t}^\eta$.
Since $\{\mathcal {F}_{t,T}^B\}_{t\geq 0}$ is decreasing and $\{\mathcal {F}_{t}^L\}_{t\geq 0}$ is increasing, the objet $\{\mathcal{F}_{t}\}_{t\geq}$ is neither increasing nor decreasing. Thus it does not a filtration.

We denote by $(H^{(i)})_{i\geq 1}$ the linear combination of so-called Teugels martingale $Y_t^{(i)}$ associated with the Lévy process $\{L_t:t\in [0,T]\}$. More
precisely
\begin{eqnarray*}
H^{(i)}_t=c_{i,i}Y_t^{(i)}+c_{i,i-1}Y_t^{(i-1)}+\cdots+c_{i,1}Y_t^{(1)},
\end{eqnarray*}
where for all $i\geq 1$, $Y_t^{(i)}=L_t^{(i)}-E[L_t^{(i)}]=L_t^{(i)}-tE[L_1^{(i)}]$.  For each $t\in[0,T],\ L_t^{(i)}$ is a power-jump processes defined as follows:
$L_t^{(1)}=L_t$ and $L_t^{(i)}=\sum_{0<s\leq t}(\triangle L_s)^i$
for $i \geq 2$. It was shown in Nualart and Schoutens \cite{Nu1}
that the coefficients $c_{i,k}$ correspond to the orthonormalization
of the polynomials
$q_{i-1}(x)=c_{i,i}x^{i-1}+c_{i,i-1}x^{i-2}+\cdots+c_{i,1}$ with
respect to the measure $\mu({\rm d}x)=x^2\nu({\rm
d}x)+\kappa^2\delta_0({\rm d}x)$:
\begin{eqnarray*}
\int_{\R}q_n(x)q_m(x)\mu(dx)=0\;\; \mbox{if}\;\; n\neq m\,\,\, \mbox{and}\;\; \int_{\R}q^2_n(x)\mu(dx)=1.
\end{eqnarray*}
We set
\begin{eqnarray*}
p_k(x)=xq_{k-1}(x).
\end{eqnarray*}
The martingales $(H^{(i)})_{i\geq
1}$ can be chosen to be pairwise strongly orthonormal martingales,
i.e. $[H^{(i)},H^{(j)}]=0, i\neq j,$ and
$\{[H^{(i)},H^{(i)}]_t-t\}_{t\geq 0}$ are uniformly integrable
martingales with initial value 0 and
$\left<H^{(i)},H^{(j)}\right>_t=\delta_{ij}t$.

\begin{remark}\rm \label{re2.1} The case of $\nu=0$
corresponds to the classic Brownian case and all non-zero degree
polynomials $q_i(x)$ will vanish, giving $H_t^{(i)}=0,i=2,3,\cdots,$
i.e. all power jump processes of order strictly greater than one
will be equal to zero. If $\nu$ only has mass at 1, we have the
Poisson case; here also $H_t^{(i)}=0,i=2,3,\cdots,$ i.e. all power
jump processes will be the same, and equal to the original Poisson
process. Both cases are degenerate in this L\'{e}vy framework.
\end{remark}

Let us introduce the following appropriate spaces:
\begin{itemize}
\item $\ell^2=\left\{x=(x_n)_{n\geq1};\ \|x\|_{\ell^2}=
\left(\sum_{n=1}^\infty|x_n|^2\right)^{1/2}<\infty\right\}$.

\item $\mathcal{H}^2$ the subspace of the $\mathcal{F}_t%
$-measurable and $\mathbb{R}$-valued processes $(Y_t)_{t\in[0,T]}$
such that $$\|Y\|^2_{\mathcal{H}^2}
=\mathbb{E}\int_0^T|Y_{t}|^{2}\,{\rm d}t<+\infty.
$$
\item $S^2$ the subspace of the $\mathbb{R}$-valued, $\mathcal{F}_t%
$-measurable, right continuous left limited (rcll, in short)
processes $(Y_t)_{t\in[0,T]}$ such that $$\|Y\|^2_{S^2}
=\mathbb{E}\left(\sup_{0\leq t\leq T}|Y_{t}|^{2}\right)<+\infty. $$

\item $\mathcal{P}^{2}(l^2)$ the space of jointly predictable processes $(Z)_{t\in[0,T]}$ taking values in $\ell^2$ such that
$$
\|Z\|^2_{\mathcal{P}^{2}(l^2)}=\mathbb{E}\int_0^T\|Z_s\|^2_{\ell^2}\,{\rm
d}s= \sum_{i=1}^\infty \mathbb{E}\int_0^T|Z_s^{(i)}|^2\,{\rm
d}s<\infty.$$
\end{itemize}
Now, we give the following assumptions:
\begin{itemize}
 \item[(H1)] The coefficients $f:[0,T]\times \Omega \times \mathbb{R}\times \ell^2
 \rightarrow \mathbb{R}$ and $g:[0,T]\times \Omega \times \mathbb{R}\times \ell^2
 \rightarrow \mathbb{R}$
satisfy, for all $t\in [0,T],y \in \mathbb{R}$ and $z\in \ell^2$,
\begin{itemize}
\item[(i)] $f(t,\cdot,y,z)$ and $g(t,\cdot,y,z)$ are $\mathcal{F}_t$-measurable,
\item[(ii)] $f(\cdot,0,0),\ g(\cdot,0,0)\in \mathcal{H}^2$;
\end{itemize}
 \item[(H2)] There exists some constants $C>0$ and $0< \alpha<1$ such
 that for every
$(t,\omega)\in [0,T]\times\Omega,(y_1,z_1),(y_2,z_2)\in
\mathbb{R}\times \ell^2$
$$|f(t,\omega,y_1,z_1)-f(t,\omega,y_2,z_2)|^2\leq C\left(|y_1-y_2|^2+\|z_1-z_2\|_{\ell^2}^2\right),$$
$$|g(t,\omega,y_1,z_1)-g(t,\omega,y_2,z_2)|^2\leq
C|y_1-y_2|^2+\alpha\|z_1-z_2\|_{\ell^2}^2;$$
 \item[(H3)] Let
$\varphi:\mathbb{R}\rightarrow (-\infty,+\infty]$ be a proper lower
semi continuous convex function satisfying $\varphi(y)\geq
\varphi(0)=0$;
 \item[(H4)] The terminal value $\xi\in L^2(\Omega,\mathcal{F}_T,\mathbb{P})$
 satisfies
\begin{eqnarray*}
\mathbb{E}\left(|\xi|^2+\varphi(\xi)\right)<\infty.
\end{eqnarray*}
\end{itemize}
Define:
\begin{description}
\item ${\rm Dom}(\varphi)=\left \{u\in \mathbb{R} :\varphi(u)<+\infty \right\},$

\item $\partial \varphi(u)= \{u^\ast\in
\mathbb{R}:\left<u^\ast,v-u\right>+\varphi(u)\leq \varphi(v),\
\mbox{for all}\ v\in \mathbb{R}\},$

\item ${\rm Dom}(\partial \varphi))= \{u\in \mathbb{R}:\partial \varphi(u) \neq
\emptyset\},$

\item ${\rm Gr}(\partial \varphi))= \{(u,u^\ast)\in \mathbb{R}^2: u\in {\rm
Dom} (\partial \varphi), u^\ast \in \partial \varphi (u)\}.$
\end{description}
Now, we introduce a multi-valued maximal monotone operator on
$\mathbb{R}$ defined by the subdifferential of the above function
$\varphi$.

For all $x\in \mathbb{R}$, define
$$\varphi_\varepsilon(x)=\min_{y}\left(\frac{1}{2}|x-y|^2+\varepsilon\varphi(y)\right)=\frac{1}{2}\left|x-J_\varepsilon(x)\right|^2
+\varepsilon \varphi(J_\varepsilon x),$$ where
$J_\varepsilon(x)=(I+\varepsilon \partial \varphi)^{-1}(x)$ is
called the resolvent of the monotone operator $A=\partial \varphi$.
Then, we have the following Proposition which appeared in Brezis
\cite{Br}.
\begin{proposition}\label{pro2.1}
\begin{enumerate}
\item[\rm(1)]  $\varphi_\varepsilon:\mathbb{R}\rightarrow \mathbb{R}$ is a convex
function with Lipschitz continuous derivatives;
\item[\rm(2)] $\forall x\in \mathbb{R},
\frac{1}{\varepsilon}D\varphi_\varepsilon(x)=\frac{1}{\varepsilon}\partial
\varphi_\varepsilon(x)=\frac{1}{\varepsilon}\left(x-J_\varepsilon(x)\right)\in
\partial \varphi(J_\varepsilon(x));$
\item[\rm(3)] $\forall x,y\in \mathbb{R},
|J_\varepsilon(x)-J_\varepsilon(y)|\leq |x-y|;$
\item[\rm(4)] $\forall x\in \mathbb{R}, 0\leq \varphi_\varepsilon(x)\leq
\left<D\varphi_\varepsilon(x),x\right>;$
\item[\rm(5)] $\forall x,y\in \mathbb{R}\ \mbox{and} \ \varepsilon,\delta>0,\,
\left<\frac{1}{\varepsilon}D\varphi_\varepsilon(x)-\frac{1}{\delta}D\varphi_\delta(y),x-y\right>\geq-\left(\frac{1}{\varepsilon}
 +\frac{1}{\delta}\right)|D\varphi_\varepsilon(x)||D\varphi_\delta(y)|$.
 \end{enumerate}
\end{proposition}
We first give the definition of BDSDEs involving subdifferential
operator of a convex function and driven by Lévy process.
 \begin{definition}\rm \label{def2.1}
We call solution of BDSDE $(\xi,f,g,\varphi)$ a triple of $(Y,U,Z)$
of progressively measurable processes  such that
\begin{enumerate}
\item $(Y,Z)\in S^2\times \mathcal{P}^2(l^2),\ U\in \mathcal{H}^2$;
\item $(Y_t,U_t)\in \partial \varphi, \ \,{\rm d}\mathbb{P}\otimes \,{\rm d}t\mbox{-a.e. on}\ [0,T];$
\item $ \displaystyle
Y_t+\int_t^TU_s\,{\rm d}s =\xi+\int_t^Tf(s,Y_s,Z_s)\,{\rm
d}s+\int_t^Tg(s,Y_s,Z_s)\,{\rm d}B_s-
\sum_{i=1}^\infty\int_t^TZ_s^{(i)}\,{\rm d}H_s^{(i)}, 0\leq t \leq
T.$
\end{enumerate}
\end{definition}
\section{Existence and uniqueness result for BDSDE driven by Lévy process}
The first result of the paper is the following theorem:
\begin{theorem}\label{thm3.1}
Assume the assumptions of {\rm(H1)--(H4)} hold. Then, the BDSDE
$(\xi,f,g,\varphi)$ has a unique solution.
\end{theorem}
For the prove of this theorem, let us consider the following BDSDEs:
\begin{eqnarray}\label{eq1}
Y_t^\varepsilon+\frac{1}{\varepsilon}\int_t^TD\varphi_\varepsilon(Y_s^\varepsilon)\,{\rm
d}s
&=&\xi+\int_t^Tf(s,Y_s^\varepsilon,Z_s^\varepsilon)\,{\rm d}s+\int_t^Tg(s,Y_s^\varepsilon,Z_s^\varepsilon)\,{\rm d}B_s\nonumber\\
&&-\sum_{i=1}^\infty\int_t^TZ_s^{\varepsilon,(i)}\,{\rm
d}H_s^{(i)},\ 0\leq t \leq T,
\end{eqnarray}
where $\varphi_\varepsilon$ is the Yosida approximation of the
operator $A=\partial\varphi$. Since
$\frac{1}{\varepsilon}D\varphi_\varepsilon(Y_s^\varepsilon)$ is
Lipschitz continuous, it is known from a recent result of Ren et al.
\cite{Ren2}, that Eq. (\ref{eq1}) has a unique solution $
(Y^\varepsilon,Z^\varepsilon)\in
 S^2\times \mathcal{P}^2(l^2)$.

Setting
$U_t^\varepsilon=\frac{1}{\varepsilon}D\varphi_\varepsilon(Y_t^\varepsilon),
\ 0\leq t \leq T,$ our aim is to prove that the sequence
$(Y^\varepsilon,U^\varepsilon,Z^\varepsilon)$ converges to a
sequence $(Y,U,Z)$ which is the desired solution of the BDSDEs.

In the sequel, $C>0$ is a constant which can change its value from
line to line. Firstly, we give a prior estimates on the solution.
\begin{lemma}\label{lam3.1}
Assume the assumptions of {\rm(H1)--(H4)} hold. Then, there exists a
constant $C_1>0$ such that for all $\varepsilon>0$
$$\mathbb{E}\left(\sup_{0\leq t \leq
T}|Y_t^\varepsilon|^2+\int_0^T\|Z_s^\varepsilon\|_{\ell^2}^2\,{\rm
d}s\right)\leq C_1.$$
\end{lemma}
\begin{proof}
Applying the It\^{o} formula to $|Y_t^\varepsilon|^2$ yields that

\begin{eqnarray}
|Y_t^\varepsilon|^2+\frac{2}{\varepsilon}
\int_t^TY_s^\varepsilon D\varphi_\varepsilon(Y_s^\varepsilon)\,{\rm
d}s
&=&|\xi|^2+2\int_t^TY_{s^-}^\varepsilon
f(s,Y_s^\varepsilon,Z_s^\varepsilon)\,{\rm d}s
+2\int_t^TY_{s^-}^\varepsilon g(s,Y_s^\varepsilon,Z_s^\varepsilon)\,{\rm d}B_s \nonumber\\
&&+\int_t^T|g(s,Y_s^\varepsilon,Z_s^\varepsilon)|^2\,{\rm
d}s-\sum_{i=1}^{\infty}\int_t^T|Z_s^{\varepsilon,(i)}|^2\,{\rm d[H^{(i)},H^{(i)}]}_s
\nonumber\\
&&-2\sum_{i=1}^\infty \int_t^TY_{s^-}^\varepsilon
Z_s^{\varepsilon,(i)}\,{\rm d}H_s^{(i)}.\label{eq2}
\end{eqnarray}
Noting that the fact $Y_s^\varepsilon
D\varphi_\varepsilon(Y_s^\varepsilon)\geq 0$ and taking expectation
on the both sides, we obtain
\begin{eqnarray}
\mathbb{E}|Y_t^\varepsilon|^2+\mathbb{E}\int_t^T\|Z_s^\varepsilon\|_{\ell^2}^2\,{\rm
d}s &\leq& \mathbb{E} |\xi|^2+2\mathbb{E}\int_t^TY_{s^-}^\varepsilon
f(s,Y_s^\varepsilon,Z_s^\varepsilon)\,{\rm d}s\nonumber\\
&&+\mathbb{E}\int_t^T|g(s,Y_s^\varepsilon,Z_s^\varepsilon)|^2\,{\rm
d}s.\label{eq4}
 \end{eqnarray}
Using the elementary inequality $2ab\leq \beta^2
a^2+\frac{b^2}{\beta^2}$ for all $a,b\geq 0$ and (H2), we get
\begin{eqnarray*}
2yf(s,y,z)&=&2y(f(s,y,z)-f(s,0,0))+2yf(s,0,0)\\
&\leq& \frac{1}{M}|y|^2+MC|y|^2+MC\|z\|_{\ell^2}^2+|y|^2+|f(s,0,0)|^2\\
&\leq
&\left(1+\frac{1}{M}+MC\right)|y|^2+|f(s,0,0)|^2+MC\|z\|_{\ell^2}^2,
\end{eqnarray*}
 \begin{eqnarray*}
|g(s,y,z)|^2&=&|g(s,y,z)-g(s,0,0)+g(s,0,0)|^2\\
&\leq& \left(1+\frac{1}{\beta}\right)|g(s,y,z)-g(s,0,0)|^2+(1+\beta)|g(s,0,0)|^2\\
&\leq&\left(1+\frac{1}{\beta}\right)C|y|^2+(1+\beta)|g(s,0,0)|^2+\alpha\left(1+\frac{1}{\beta}\right)\|z\|_{\ell^2}^2.
\end{eqnarray*}
Choosing $M=\frac{1-\alpha}{2C},\beta=\frac{3\alpha}{1-\alpha}$, it
follows from (\ref{eq4}) that
\begin{eqnarray*}
&&\mathbb{E}|Y_t^\varepsilon|^2+\frac{1-\alpha}{6}\mathbb{E}\int_t^T\|Z_s^\varepsilon\|_{\ell^2}^2\,{\rm
d}s\\
&\leq& C\mathbb{E}\left( |\xi|^2+\int_t^T|Y_s^\varepsilon|^2\,{\rm
d}s+\int_0^T|f(s,0,0)|^2\,{\rm d}s+\int_0^T|g(s,0,0)|^2\,{\rm
d}s\right).
\end{eqnarray*}
Gronwall inequality and Bulkholder-Davis-Gundy inequality show the desired result.
\end{proof}
\begin{lemma}\label{lam3.2}
Assume the assumptions of {\rm(H1)--(H4)} hold. Then, there exists a
constant $C_2>0$ such that
\begin{description}
\item $(i)$ $\displaystyle\mathbb{E}\int_0^T\left(\frac{1}{\varepsilon}\left|D\varphi_\varepsilon(Y_s^\varepsilon)\right|\right)^2\,{\rm
d}s\leq C_2;$
\item $(ii)$ $\displaystyle\mathbb{E}\varphi\left(J_\varepsilon(Y_t^\varepsilon)\right)\leq C_2;$
\item $(iii)$ $\displaystyle\mathbb{E}|Y_t^\varepsilon-J_\varepsilon(Y_t^\varepsilon)|^2\leq \varepsilon^2 C_2.$
\end{description}
\end{lemma}
\begin{proof}
(i) Given an equidistant partition of interval $[0,T]$ such that
$0=t_0<t_1<t_2<\cdots<t_n=T$ and $t_{i+1}-t_i=\frac{1}{n}$, the
subdifferential inequality shows
$$\varphi_\varepsilon(Y_{t_{i+1}}^\varepsilon)\geq
\varphi_\varepsilon(Y_{t_{i}}^\varepsilon)+(Y_{t_{i+1}}^\varepsilon-Y_{t_{i}}^\varepsilon)
D\varphi_\varepsilon(Y_{t_{i}}^\varepsilon).$$ From (\ref{eq1}), we
obtain

\begin{eqnarray*}
\varphi_\varepsilon(Y_{t_i}^\varepsilon)+\frac{1}{\varepsilon}
\int_{t_i}^{t_{i+1}}D\varphi_\varepsilon(Y_{t_i}^\varepsilon)
D\varphi_\varepsilon(Y_s^\varepsilon)\,{\rm d}s
&\leq&\varphi_\varepsilon(Y_{t_{i+1}}^\varepsilon)
+\int_{t_i}^{t_{i+1}}D\varphi_\varepsilon(Y_{t_i}^\varepsilon)f(s,Y_s^\varepsilon,Z_s^\varepsilon)\,{\rm
d}s
\nonumber\\
&&+\int_{t_i}^{t_{i+1}}D\varphi_\varepsilon(Y_{t_i}^\varepsilon) g(s,Y_s^\varepsilon,Z_s^\varepsilon)\,{\rm d}B_s \nonumber\\
&&-2\sum_{j=1}^\infty
\int_{t_i}^{t_{i+1}}D\varphi_\varepsilon(Y_{t_i}^\varepsilon)
(Z_s^{\varepsilon})^{(j)}\,{\rm d}H_s^{(j)}.
 \end{eqnarray*}
Summing up the above formula over $i$ and letting $n\rightarrow
\infty,$ we obtain

\begin{eqnarray*}
&&\varphi_\varepsilon(Y_{t^-}^\varepsilon)+\frac{1}{\varepsilon}
\int_0^T| D\varphi_\varepsilon(Y_{s^-}^\varepsilon)|^2\,{\rm d}s\nonumber\\
&\leq&\varphi_\varepsilon(\xi)
+\int_0^TD\varphi_\varepsilon(Y_{s^-}^\varepsilon)f(s,Y_s^\varepsilon,Z_s^\varepsilon)\,{\rm
d}s +\int_0^TD\varphi_\varepsilon(Y_{s^-}^\varepsilon) g(s,Y_s^\varepsilon,Z_s^\varepsilon)\,{\rm d}B_s \nonumber\\
&&-2\sum_{j=1}^\infty \int_0^TD\varphi_\varepsilon(Y_{s^-}^\varepsilon)
(Z_s^{\varepsilon})^{(j)}\,{\rm d}H_s^{(j)}.\label{eq6}
 \end{eqnarray*}
Taking expectation on the both sides, we get
\begin{eqnarray}
&&\mathbb{E}\varphi_\varepsilon(Y_t^\varepsilon)+\frac{1}{\varepsilon}\mathbb{E}
\int_0^T| D\varphi_\varepsilon(Y_s^\varepsilon)|^2\,{\rm d}s
\leq\mathbb{E}\varphi_\varepsilon(\xi)
+\mathbb{E}\int_0^TD\varphi_\varepsilon(Y_s^\varepsilon)f(s,Y_s^\varepsilon,Z_s^\varepsilon)\,{\rm
d}s.\label{eq7}
 \end{eqnarray}
For
\begin{eqnarray*}\label{eq8}D\varphi_\varepsilon(y) f(s,y,z) &\leq&\frac{1}{2\varepsilon}
\left|D\varphi_\varepsilon(y)\right|^2+\frac{\varepsilon}{2}|f(s,y,z)|^2 \nonumber\\
&\leq&\frac{1}{2\varepsilon}
|D\varphi_\varepsilon(y)|^2+\varepsilon(|f(s,y,z)-f(s,0,0)|^2+|f(s,0,0)|^2)\nonumber\\
&\leq&\frac{1}{2\varepsilon} |D\varphi_\varepsilon(y)|^2+\varepsilon
C|y|^2 +\varepsilon C\|z\|_{\ell^2}^2+\varepsilon |f(s,0,0)|^2,
 \end{eqnarray*}
the fact that $\varphi_\varepsilon(Y_t^\varepsilon)\geq 0$ and
$\varphi_\varepsilon(\xi)\leq \varepsilon \varphi(\xi),$ we obtain
\begin{eqnarray*}&&\frac{1}{2\varepsilon}\mathbb{E} \int_t^T|
D\varphi_\varepsilon(Y_s^\varepsilon)|^2\,{\rm d}s\leq
C\mathbb{E}\left(\varphi(\xi)+\int_0^T|f(s,0,0)|^2\,{\rm
d}s+T\sup_{0\leq t\leq
T}|Y_t^\varepsilon|^2+\int_0^T\|Z_t^\varepsilon\|_{\ell^2}^2\,{\rm
d}t\right).\end{eqnarray*} Lemma \ref{lam3.1} shows the desired
result.
\\
(ii) From (\ref{eq7}), we obtain
$$\mathbb{E}\varphi_\varepsilon(Y_t^\varepsilon)\leq \varepsilon
\mathbb{E}\varphi(\xi)
+\frac{1}{2\varepsilon}\mathbb{E}\int_t^T|D\varphi_\varepsilon(Y_s^\varepsilon)|^2\,{\rm
d}s+\varepsilon
\mathbb{E}\int_t^T|f(s,Y_s^\varepsilon,Z_s^\varepsilon)|^2\,{\rm
d}s.$$
 Using
$\varphi(J_\varepsilon(Y_t^\varepsilon))\leq \frac{1}{\varepsilon}
\varphi_\varepsilon(Y_t^\varepsilon)$ and $(i)$, we obtain $(ii)$.
\\
The last part of the Lemma simply follows from the fact that
\begin{eqnarray*}
|x-J_{\varepsilon}(x)|=2\varphi_{\varepsilon}(x)-2\varepsilon\varphi(J_{\varepsilon}(x)).
\end{eqnarray*}
\end{proof}

In what follows, we aim to show that $(Y^\varepsilon,Z^\varepsilon)$
is a Cauchy sequence in  $S^2\times \mathcal{P}^2(l^2)$.
\begin{lemma}\label{lam3.3}
Assume the assumptions of {\rm(H1)--(H4)} hold. Then, there exists a
constant $C_3$ such that for all $\varepsilon, \delta>0$
$$\mathbb{E}\left(\sup_{0\leq t\leq
T}|Y_t^\varepsilon-Y_t^\delta|^2+\int_0^T\|Z_t^\varepsilon-Z_t^\delta\|_{\ell^2}^2\,{\rm
d}t\right)\leq C_3(\varepsilon+\delta).$$
\end{lemma}
\begin{proof}
Applying the It\^{o} formula to $|Y_t^\varepsilon-Y_t^\delta|^2$
yields that
\begin{eqnarray}|Y_t^\varepsilon-Y_t^\delta|^2&=&-2
\int_t^T(Y_s^\varepsilon-Y_s^\delta)
\left(\frac{1}{\varepsilon}D\varphi_\varepsilon(Y_s^\varepsilon)\,{\rm
d}s
-\frac{1}{\delta}D\varphi_\delta(Y_s^\delta)\right)\,{\rm d}s\nonumber\\
&&+2\int_t^T(Y_s^\varepsilon-Y_s^\delta)(f(s,Y_s^\varepsilon,Z_s^\varepsilon)-f(s,Y_s^\delta,Z_s^\delta))\,{\rm d}s\nonumber\\
&&+2\int_t^T(Y_s^\varepsilon-Y_s^\delta)(g(s,Y_s^\varepsilon,Z_s^\varepsilon)-g(s,Y_s^\delta,Z_s^\delta))\,{\rm d}B_s\nonumber\\
&&+\int_t^T|g(s,Y_s^\varepsilon,Z_s^\varepsilon)-g(s,Y_s^\delta,Z_s^\delta)|^2\,{\rm
d}s-
\int_t^T\|Z_s^\varepsilon-Z_s^\delta\|_{\ell^2}^2\,{\rm d}s\nonumber\\
&&-2\sum_{i=1}^\infty \int_t^T(Y_s^\varepsilon-Y_s^\delta)
(Z_s^{\varepsilon,(i)}-Z_s^{\delta,(i)})\,{\rm d}H_s^{(i)}.
\end{eqnarray} Taking expectation, we obtain

$\displaystyle\mathbb{E}|Y_t^\varepsilon-Y_t^\delta|^2+
\mathbb{E}\int_t^T\|Z_s^\varepsilon-Z_s^\delta\|_{\ell^2}^2\,{\rm
d}s$
\begin{eqnarray}&=&-2
\mathbb{E}\int_t^T(Y_s^\varepsilon-Y_s^\delta)
\left(\frac{1}{\varepsilon}D\varphi_\varepsilon(Y_s^\varepsilon)\,{\rm d}s-\frac{1}{\delta}D\varphi_\delta(Y_s^\delta)\right)\,{\rm d}s\nonumber\\
&&+2\mathbb{E}\int_t^T(Y_s^\varepsilon-Y_s^\delta)
(f(s,Y_s^\varepsilon,Z_s^\varepsilon)-f(s,Y_s^\delta,Z_s^\delta))\,{\rm d}s\nonumber\\
&&+\mathbb{E}\int_t^T|g(s,Y_s^\varepsilon,Z_s^\varepsilon)-g(s,Y_s^\delta,Z_s^\delta)|^2\,{\rm
d}s.
 \end{eqnarray}
Using the elementary inequality $2ab\leq \beta^2
a^2+\frac{b^2}{\beta^2}$ for all $a,b\geq 0$ and (H2), we get

$(Y_s^\varepsilon-Y_s^\delta)(f(s,Y_s^\varepsilon,Z_s^\varepsilon)-f(s,Y_s^\delta,Z_s^\delta))$
$$\leq
\frac{2C}{1-\alpha}|Y_s^\varepsilon-Y_s^\delta|^2+\frac{1-\alpha}{2}|Y_s^\varepsilon-Y_s^\delta|^2+
\frac{1-\alpha}{2}\|Z_s^\varepsilon-Z_s^\delta\|_{\ell^2}^2$$ and
$$|g(s,Y_s^\varepsilon,Z_s^\varepsilon)-g(s,Y_s^\delta,Z_s^\delta)|^2\leq
C|Y_s^\varepsilon-Y_s^\delta|^2+
\alpha\|Z_s^\varepsilon-Z_s^\delta\|_{\ell^2}^2. $$
 Noting (5) of Proposition\ref{pro2.1}, we obtain

$\displaystyle\mathbb{E}|Y_t^\varepsilon-Y_t^\delta|^2+
\frac{1-\alpha}{2}\mathbb{E}\int_t^T\|Z_s^\varepsilon-Z_s^\delta\|_{\ell^2}^2\,{\rm
d}s$
\begin{eqnarray}&\leq &C\varepsilon
\mathbb{E}\int_t^T|Y_s^\varepsilon-Y_s^\delta|^2\,{\rm
d}s\nonumber\\&&+2\left(\frac{1}{\varepsilon}+\frac{1}{\delta}\right)\mathbb{E}\int_t^T
|D\varphi_\varepsilon(Y_s^\varepsilon)||D\varphi_\delta(Y_s^\delta))|\,{\rm
d}s.
 \end{eqnarray}
Lemma \ref{lam3.2} shows that
$$2\left(\frac{1}{\varepsilon}+\frac{1}{\delta}\right)\mathbb{E}\int_t^T
|D\varphi_\varepsilon(Y_s^\varepsilon)||D\varphi_\delta(Y_s^\delta))|\,{\rm
d}s\leq (\varepsilon+\delta)C.$$ So, we can obtain
$$\mathbb{E}|Y_t^\varepsilon-Y_t^\delta|^2+
\mathbb{E}\int_t^T\|Z_s^\varepsilon-Z_s^\delta\|_{\ell^2}^2\,{\rm
d}s\leq C\mathbb{E}\int_t^T|Y_s^\varepsilon-Y_s^\delta|^2\,{\rm
d}s+C(\varepsilon+\delta).$$
 The Gronwall inequality shows that
$$\sup_{0\leq t \leq T}\mathbb{E}|Y_t^\varepsilon-Y_t^\delta|^2+
\mathbb{E}\int_t^T\|Z_s^\varepsilon-Z_s^\delta\|_{\ell^2}^2\,{\rm
d}s\leq C(\varepsilon+\delta).$$
 The Bulkholder-Davis-Gundy inequality shows the desired
result.
\end{proof}
\subsubsection*{Proof of Theorem \ref{thm3.1}}
{\it Existence.} Lemma \ref{lam3.3} shows that
$(Y^\varepsilon,Z^\varepsilon)$ is a Cauchy sequence in $S^2\times
\mathcal{P}^2(l^2)$. Denoting its limit by $(Y,Z)$, then it follows from Lemma \ref{lam3.2} $(Y,Z)
\in S^2\times \mathcal{P}^2(l^2)$.
For each $\varepsilon\geq 0,$ define $U_t^\varepsilon
=\frac{1}{\varepsilon}D\varphi_\varepsilon(Y_t^\varepsilon)$ and
$\bar{U}_t^\varepsilon=\int_0^t U_s^\varepsilon \,{\rm d}s$. Therefore
(\ref{eq1}) and Lemma \ref{lam3.3} yield that for all $\varepsilon,\, \delta>0$
\begin{eqnarray*}
\mathbb{E}\left(\sup_{0\leq t \leq
T}|\bar{U}_t^\varepsilon-\bar{U}_t^\delta|^2\right)\leq
C\mathbb{E}\left(\sup_{0\leq t \leq
T}|Y_t^\varepsilon-Y_t^\delta|^2+\int_0^T\|Z_t^\varepsilon-Z_t^\delta\|_{\ell^2}^2\,{\rm
d}t\right),
\end{eqnarray*}
which shows that $(\bar{U}^\varepsilon)$ is a Cauchy sequence.
Hence, there exists a measurable process $\bar{U_t}$ such that
\begin{eqnarray*}
\lim_{\varepsilon\rightarrow 0}\mathbb{E}\left(\sup_{0\leq t \leq
T}|\bar{U}_t^\varepsilon-\bar{U_t}|^2\right)=0.
\end{eqnarray*}
 Furthermore, Lemma \ref{lam3.2} $(i)$ shows that
\begin{eqnarray*}
\sup_{\varepsilon}\mathbb{E}\int_0^T|U_t^\varepsilon|^2\,{\rm d}t=
\sup_{\varepsilon}\mathbb{E}\int_0^T\left(\frac{1}{\varepsilon}|D\varphi_\varepsilon(Y_t^\varepsilon)|\right)^2\,{\rm
d}t<\infty,
\end{eqnarray*}
 which shows that $\bar{U}_t^\varepsilon$ is bounded in the space $L^2(\Omega,H^1[0,T])$, and
$(\bar{U}^\varepsilon)_{\varepsilon}$ converges weakly to a
limit in that space and the limit is necessarily $\bar{U}.$ In
particular, $\bar{U}$ is absolutely continuous. So, there
exists a measurable process $(U_t)_{0\leq t \leq T }\in
\mathcal{H}^2$ such that $\bar{U}_t=\int_0^t U_s\,{\rm d}s.$

Next, we show that $(Y_t,U_t)\in \partial \varphi, \,{\rm
d}\mathbb{P}\otimes \,{\rm d}t\mbox{-a.e. on}\ [0,T].$ Moreover, with the help of Lemma 5.8 in \cite{Geg}, and for all $0\leq a<b\leq T, V\in
\mathcal{H}^2([a,b])$, we obtain
\begin{eqnarray*}
\int_a^b U_t^\varepsilon(V_t-Y_t^\varepsilon)\,{\rm d}t\rightarrow
\int_a^b U_t(V_t-Y_t)\,{\rm d}t,\ \mbox{as}\ \varepsilon\rightarrow
0
\end{eqnarray*}
in probability. In particular we have
\begin{eqnarray*}
 \int_a^b
U_t^\varepsilon(J_\varepsilon(Y_t^\varepsilon)-Y_t^\varepsilon)\,{\rm
d}t\rightarrow 0,\ \mbox{as}\ \varepsilon\rightarrow 0.
\end{eqnarray*}
which together with Proposition \ref{pro2.1} provides that $U_t^\varepsilon\in \partial
\varphi(J_\varepsilon(Y_t^\varepsilon))$ and
\begin{eqnarray*}
\int_a^bU_t^\varepsilon(V_t-J_\varepsilon(Y_t^\varepsilon)\,{\rm
d}t+\int_a^b\varphi(J_\varepsilon(Y_t^\varepsilon))\,{\rm d}t\leq
\int_a^b\varphi(V_t)\,{\rm d}t.
\end{eqnarray*}
Taking the $\liminf$ in probability in the above inequality, we
obtain
\begin{eqnarray*}
\int_a^bU_t(V_t-Y_t)\,{\rm d}t+\int_a^b\varphi(Y_t)\,{\rm d}t\leq
\int_a^b\varphi(V_t)\,{\rm d}t.
\end{eqnarray*}
Since $a, b$ and the process $V$ are arbitrary, this shows that
\begin{eqnarray*}
U_t(V_t-Y_t)+\varphi(Y_t)\leq \varphi(V_t), \,{\rm d}\mathbb{P}\otimes {\rm d}t\mbox{-a.e. on}\
[0,T].
\end{eqnarray*}
Taking limit on the both sides of (\ref{eq1}), we obtain
the existence of the solution.

\textsl{Uniqueness.} Let $(Y_t,U_t,Z_t)_{0\leq t \leq T}$ and
$(Y_t^\prime,U_t^\prime,Z_t^\prime)_{0\leq t \leq T}$ be two
solutions of BDSDEs associated with $(\xi,f,g,\varphi)$. Define
$$(\Delta Y_t, \Delta U_t,\Delta Z_t)_{0\leq t \leq
T}=(Y_t-Y_t^\prime,U_t-U_t^\prime,Z_t-Z_t^\prime)_{0\leq t \leq
T}.$$ Applying the It\^{o} formula to $|\Delta Y_t|^2$ shows that

 $\displaystyle\mathbb{E}|\Delta Y_t|^2+2\mathbb{E}\int_t^T\Delta U_s \Delta Y_s\,{\rm d}s
 + \mathbb{E}\int_t^T\|\Delta Z_t\|_{\ell^2}^2\,{\rm d}s$
\begin{eqnarray}
&=&2\mathbb{E}\int_t^T\Delta Y_s
[f(s,Y_s,Z_s)-f(s,Y_s^\prime,Z_s^\prime)]\,{\rm d}s\nonumber\\
&&+\mathbb{E}\int_t^T
|g(s,Y_s,Z_s)-g(s,Y_s^\prime,Z_s^\prime)|^2\,{\rm d}s.
 \end{eqnarray}
Since $\partial \varphi$ is monotone, we obtain $$\Delta U_s \Delta
Y_s\geq 0,\ \,{\rm d}\mathbb{P}\otimes {\rm d}t\mbox{-a.e.}$$
 Further, as the same procedure as Lemma \ref{lam3.3}, we obtain
 $$\mathbb{E}|\Delta Y_t|^2+\mathbb{E}\int_t^T\|\Delta Z_t\|_{\ell^2}^2\,{\rm d}s
 \leq C\mathbb{E}\int_t^T|\Delta Y_s|^2\,{\rm d}s
 +\frac{1}{2}\mathbb{E}\int_t^T\|\Delta Z_s\|_{\ell^2}^2\,{\rm d}s.$$
The Gronwall inequality shows the uniqueness of the solution.

\section {Stochastic viscosity solutions of multivalued SPDIEs}

In this section, we derive the existence of the stochastic viscosity
solution of a class of multivalued SPDIE (\ref{i11}) via BDSDE with
subdifferential operator and driven by Lévy process studied in the previous section.

\subsection {Notion of stochastic viscosity solution of multivalued  SPDIEs}

Let us recall ${\bf F}^{B}=\{\mathcal{F}_{t,T}^{B}\}_{0\leq t\leq
T}$ be the filtration generated by $B$. The objet
${\mathcal{M}}^{B}_{0,T}$ denotes all the ${\bf F}^{B}$-stopping
times $\tau$ such $0\leq \tau\leq T$, a.s. and
${\mathcal{M}}^{B}_{\infty}$ is the set of all almost surely finite
${\bf F}^{B}$-stopping times. For generic Euclidean spaces $E$ and
$E_{1}$, we state those spaces:
\begin{enumerate}
\item The symbol $\mathcal{C}^{k,n}([0,T]\times
E; E_{1})$ stands for the space of all $E_{1}$-valued functions
defined on $[0,T]\times E$ which are $k$-times continuously
differentiable in $t$ and $n$-times continuously differentiable in
$x$, and $\mathcal{C}^{k,n}_{b}([0,T]\times E; E_{1})$ denotes the
subspace of $\mathcal{C}^{k,n}([0,T]\times E; E_{1})$ in which all
functions have uniformly bounded partial derivatives.
\item For any sub-$\sigma$-field $\mathcal{G} \subseteq
\mathcal{F}_{T}^{B}$, $\mathcal{C}^{k,n}(\mathcal{G},[0,T]\times E;
E_{1})$ (resp.\, $\mathcal{C}^{k,n}_{b}(\mathcal{G},[0,T]\times E;
E_{1})$) denotes the space of all $\mathcal{C}^{k,n}([0,T]\times E;
E_{1})$  (resp.\, $\mathcal{C}^{k,n}_{b}([0,T]\times
E;E_{1})$-valued random variable that are
$\mathcal{G}\otimes\mathcal{B}([0,T]\times E)$-measurable;
\item $\mathcal{C}^{k,n}({\bf F}^{B},[0,T]\times E; E_{1})$
(resp.$\mathcal{C}^{k,n}_{b}({\bf F}^{B},[0,T]\times E; E_{1})$) is
the space of all random fields $\varphi\in
\mathcal{C}^{k,n}({\mathcal{F}}_{T},[0,T]\times E; E_{1}$ (resp.
$\mathcal{C}^{k,n}({\mathcal{F}}_{T},[0,T]\times E; E_{1})$, such
that for fixed $x\in E$ and $t\in [0,T]$, the mapping
$\displaystyle{\omega\rightarrow \alpha(t,\omega,x)}$ is ${\bf
F}^{B}$-progressively measurable.
\item For any sub-$\sigma$-field $\mathcal{G} \subseteq
\mathcal{F}^{B}$ and a real number $ p\geq 0$,
$L^{p}(\mathcal{G};E)$ denotes the set of all $E$-valued,
$\mathcal{G}$-measurable random variable $\xi$ such that $
\E|\xi|^{p}<\infty$.
\end{enumerate}
Furthermore, regardless of the dimension, we denote by
$\left<\cdot,\cdot\right>$ and $|\cdot|$ the inner product and norm
in $E$ and $E_1$, respectively. For
$(t,x,y)\in[0,T]\times\R^{d}\times\R$, we denote
$D_{x}=(\frac{\partial}{\partial
x_{1}},....,\frac{\partial}{\partial x_{d}}),\,
D_{xx}=(\partial^{2}_{x_{i}x_{j}})_{i,j=1}^{d}$,
$D_{y}=\frac{\partial}{\partial y}, \,\
D_{t}=\frac{\partial}{\partial t}$. The meaning of $D_{xy}$ and
$D_{yy}$ is then self-explanatory.\newline The coefficients
\begin{eqnarray*}
f&:&\Omega\times[0,T]\times\R^d\times\R\times\ell^{2}\to\R\\
g&:&\Omega\times[0,T]\times\R^d\times\R\to\R\\
\sigma&:&\R^{d}\to\R^{d}\\
u_0&:&\R^d\to \R,
\end{eqnarray*}
satisfying assumptions:
\begin{description}
\item $
(\rm{ H5})\ \left\{
\begin{array}{l}
|f(t,x,y,z)|\leq K(1+|x|+|y|+\|z\|),\\\\
|u_0(x)|+|\varphi(u_0(x))|\leq K(1+|x|).
\end{array}\right.
$
\item
$ (\rm{ H6})\ \left\{
\begin{array}{l}
\|\sigma(x)-\sigma(x')\|\leq K|x-x'|,\\\\
|f(t,x,y,z)-f(t,x,y',z')|\leq K(|y-y'|+\|z-z'\|_{\ell^2}).
\end{array}\right.
$
\item $(\rm{H7})$ The function
$g\in{\mathcal{C}}_{b}^{0,2,3}([0,T]\times\R^d\times\R;\R)$.
\end{description}
The definition of stochastic viscosity solution to MSPDIE (\ref{i1})
use the stochastic sub-and super-jets introduced by Buckdahn and Ma
\cite{BMa}. Let us recall the following needed definitions.
\begin{definition}
Let $\tau\in {\mathcal{M}}^{B}_{0,T}$, and
$\xi\in\mathcal{F}_{\tau}$. We say that a sequence of random
variables $(\tau_k,\xi_k)$ is a $(\tau,\xi)$-approximating sequence
if for all $k$, $(\tau_k,\xi_k)\in{\mathcal{M}}^{B}_{\infty}\times
L^{2}(\mathcal{F}_{\tau},\R^d)$ such that
\begin{itemize}
\item [(i)] $\xi_k\rightarrow\xi$  in probability;
\item [(ii)] either $\tau_k\uparrow\tau$ a.s., and $\tau_k<\tau$ on the set $\{\tau>0\}$; or $\tau_k\downarrow\tau$ a.s., and $\tau_k>\tau$ on the set $\{\tau<T\}$.
\end{itemize}
\end{definition}
\begin{definition}
Let $(\tau,\xi)\in \mathcal{M}_{0,T}^B\times
L^2\left(\mathcal{F}^{B}_{\tau}; \R^d\right)$ and $u\in
\mathcal{C}\left(\mathbf{F}^B, [0,T]\times \R^d\right)$. We denote
by $\mathcal{J}^{1,2,+}_{g} u(\tau,\xi)$ the stochastic $g$-superjet
of $u$ at $(\tau,\xi)$ the set of
$\R\times\R^d\times\mathcal{S}(n)$-valued and
$\mathcal{F}_{\tau}^B$-measurable random vector $(a, p,X)$
($\mathcal{S}(d)$ is the set of all symmetric $d\times d$ matrix)
which is such that for all $(\tau, \xi)$-approximating sequence
$(\tau_k,\xi_k)$, we have
\begin{eqnarray}
u(\tau_k,\xi_k)&\leq& u(\tau,\xi)+a(\tau_k-\tau)+b(B_{\tau_k}-B_{\tau})+\frac{c}{2}(B_{\tau_k}-B_{\tau})^2+\langle p,\xi_k-\xi\rangle\nonumber\\
&&+\langle q,\xi_k-\xi\rangle(B_{\tau_k}-B_{\tau})+\frac{1}{2}\langle X(\xi_k-\xi),\xi_k-\xi\rangle\nonumber\\
&&+o(|\tau_k-\tau|)+o(|\xi_k-\xi|^2).\label{jet}
\end{eqnarray}
The $\mathcal{F}_{\tau}^B$-measurable random vector $(b,c,q)$ taking
values in $\R\times\R^d\times\R^d$ is defined by
\begin{eqnarray*}
\left\{
\begin{array}{l}
b=g(\tau,\xi,u(\tau,\xi)),\;\;\; c=(g\partial_ug)(\tau,\xi,u(\tau,\xi))\\\\
q=\partial_xg(\tau,\xi,u(\tau,\xi))+\partial_ug(\tau,\xi,u(\tau,\xi))p.
\end{array}\right.
\end{eqnarray*}
Similarly, $\mathcal{J}^{1,2,-}_{g} u(\tau,\xi)$ denotes the set of
all stochastic $g$-subjet of $u$ at $(\tau,\xi)$ if the inequality
in (\ref{jet}) is reversed.
\end{definition}
\begin{remark}
Let us note that $\partial\varphi(y)=
[\varphi'_l(y),\varphi'_r(y)]$, for every $y\in {\rm Dom}(\varphi)$,
where $\varphi'_l(y)$ and $\varphi'_r(y)$ denote the left and right
derivatives of $\varphi$.
\end{remark}
In order to simplify notation in the definition of the notion
of stochastic viscosity solution of multivalued SPDIEs, we set
\begin{eqnarray*}
V_{f}(\tau,\xi,a,p,X)&=&-a-\frac{1}{2}{\rm
Trace}(\sigma\sigma^*(\xi)X)-m_1\langle
p,\sigma(\xi)\rangle-\frac{1}{2}\int_{\R}\langle X\sigma(\xi),\sigma(\xi)\rangle y^2 \nu({\rm d}y)\\
&&-f\left(\tau,\xi,u(\tau,\xi),\int_{\R}\langle
p,\sigma(\xi)y\rangle p_k(y)\nu({\rm d}y)\right).
\end{eqnarray*}
\begin{definition}\label{defvisco}
(1) A random field $u \in \mathcal{C}\left(\mathbf{F}^B, [0,T]\times
\R^d\right)$   which satisfies $u\left(T,x\right)=u_0\left(x\right)$, for all $x\in \R^d$, is
called a stochastic viscosity subsolution of MSPDIE (\ref{i1}) if
\begin{eqnarray*}
u(\tau,\xi)&\in & {\rm Dom}(\varphi),\;\;\;\;\;\; \forall\;
(\tau,\xi)\in\mathcal{M}_{0,T}^B\times
L^2\left(\mathcal{F}^{B}_{\tau};\R^d\right),\;\;\; \P\mbox{-a.s.},
\end{eqnarray*}
and at any $(\tau,\xi)\in\mathcal{M}_{0,T}^B\times
L^2\left(\mathcal{F}^{B}_{\tau};\R^d\right)$, for any $(a,
p,X)\in\mathcal{J}^{1,2,+}_{g} u(\tau,\xi)$, it hold $\P$-a.s.
\begin{eqnarray}\label{E:def1}
V_{f}(\tau,\xi,a,p,X)+\varphi'_l(u(\tau,\xi)-\frac{1}{2}(g\partial_ug)(\tau,\xi,u(\tau,\xi))\leq
0;
\end{eqnarray}
(2) A random field $u \in \mathcal{C}\left(\mathbf{F}^B,
[0,T]\times\R^d\right)$   which satisfies that
$u\left(T,x\right)=u_0\left(x\right)$, for all $x\in\R^d$, is called
a stochastic viscosity supersolution of MSPDIE (\ref{i11}) if
\begin{eqnarray*}
u(\tau,\xi)&\in & {\rm Dom}(\varphi),\;\;\;\;\;\;
\forall\;(\tau,\xi)\in\mathcal{M}_{0,T}^B\times
L^2\left(\mathcal{F}^{B}_{\tau};\R^d\right),\;\;\; \P\mbox{-a.s.},
\end{eqnarray*}
and at any $(\tau,\xi)\in\mathcal{M}_{0,T}^B\times
L^2\left(\mathcal{F}^{B}_{\tau};\R^d\right)$, for any $(a,
p,X)\in\mathcal{J}^{1,2,-}_{g} u(\tau,\xi)$, it hold $\P$-a.s.
\begin{eqnarray}\label{E:def01}
V(\tau,\xi,a,p,X)+\varphi'_r(u(\tau,\xi)-\frac{1}{2}(g\partial_ug)(\tau,\xi,u(\tau,\xi))\geq
0;
\end{eqnarray}
(3) A random field $u \in \mathcal{C}\left(\mathbf{F}^B, [0,T]\times
\R^d\right)$ is called a stochastic viscosity solution of MSPDIE
(\ref{i1}) if it is both a stochastic viscosity subsolution and a
stochastic viscosity supersolution.
\end{definition}
\begin{remark}
Observe that if $f$ is deterministic and $g\equiv 0$, Definition\,
\ref{defvisco} becomes the generalization of the definition of (deterministic)
viscosity solution of MPDIE given by N'zi and
 Ouknine in \cite{Mn2}.
\end{remark}
To end this section, we state the notion of random viscosity
solution which will be a bridge link to the stochastic viscosity
solution and its deterministic counterpart.
\begin{definition}
A random field $u\in C({\bf F}^B, [0,T]\times\R^d)$ is called an
$\omega$-wise viscosity solution if for $\P$-almost all $\omega\in
\Omega,\;  u(\omega,\cdot,\cdot)$ is a (deterministic) viscosity
solution of MSPDIE \eqref{i11}.
\end{definition}

\subsection{Doss-Sussmann transformation}
In this section, using the Doss Sussmann transformation, our aim is
to convert a multi-valued SPDIE to a PDIE with random coefficients so
that the stochastic viscosity solution can be studied
$\omega$-wisely. We first establish the link between the $g$-super
or sub jet of $u$ the solution to multi-valued SPDIE and super or sub
jet of $v$ solution to the converter PDIE with random coefficients.
For instance let us consider the stochastic flow $\eta\in C({\bf
F}^B, [0, T]\times\R^d\times\R)$, unique solution of the following
stochastic differential equation in the Stratonovich sense:
\begin{eqnarray}
\eta(t,x,y)&=&y+\int_t^T\langle g(s,x,\eta(s,x,y)), \circ {\rm
d}B_s\rangle,\label{p1}
\end{eqnarray}
where (\ref{p1}) should be viewed as going from $T$ to $t$ (i.e $y$
should be understood as the initial value). Under the assumption
$({H7})$, the mapping $y\mapsto \eta(t,x,y)$ defines a
diffeomorphism for all $(t,x),\; \P$-a.s. such that its $y$-inverse
$\varepsilon(t,x,y)$ is the solution to the following first-order
SPDE:
\begin{eqnarray*}
\varepsilon(t,x,y)=y-\int_t^T\langle D_y\varepsilon(s,x,y),\,
g(s,x,\eta(s,x,y)) \circ {\rm d}B_s\rangle.
\end{eqnarray*}
We refer the reader to their paper \cite{BM1} for a lucid discussion
on this topic. We have
\begin{proposition}\label{Propo}
Assume that the assumptions {\rm(H1)--(H7)} hold. If for
$(\tau,\xi)\in \mathcal{M}_{0,T}^B\times
L^2\left(\mathcal{F}^{B}_{\tau}; \R^d\right)$,
$u\in\mathcal{C}\left(\mathbf{F}^B, [0,T]\times\R^d\right)$ and
$(a_u,X_u,p_u)$ belongs to $\mathcal{J}^{1,2,+}_{g} u(\tau,\xi),$
then $(a_v,X_v,p_v)$ belongs to $\mathcal{J}^{1,2,+}_{0}
v(\tau,\xi)$, with $v(\cdot,\cdot) = \varepsilon(\cdot,\cdot,
u(\cdot,\cdot))$ and
\begin{eqnarray*}
\left\{
\begin{array}{l}
a_v=D_y\varepsilon(\tau,\xi,u(\tau,\xi))a_u\\\\
p_v=D_y\varepsilon(\tau,\xi,u(\tau,\xi))p_u+D_x\varepsilon(\tau,\xi,u(\tau,\xi))\\\\
X_v=D_y\varepsilon(\tau,\xi,u(\tau,\xi))X_u+2D_{xy}\varepsilon(\tau,\xi,u(\tau,\xi))p^*_u
\\~~~~~~~+D_{xx}\varepsilon(\tau,\xi,u(\tau,\xi))
+D_{yy}\varepsilon(\tau,\xi,u(\tau,\xi))p_up_u^*.
\end{array}\right.
\end{eqnarray*}
Conversely, if for $(\tau,\xi)\in \mathcal{M}_{0,T}^B\times
L^2\left(\mathcal{F}^{B}_{\tau}; \R^d\right)$, $v\in
\mathcal{C}\left(\mathbf{F}^B, [0,T]\times\R^d\right)$ and
$(a_v,X_v,p_v)\in\mathcal{J}^{1,2,+}_{0} v(\tau,\xi)$, then
$(a_u,X_u,p_u)\in\mathcal{J}^{1,2,+}_{g} u(\tau,\xi)$ with
$u(\cdot,\cdot) = \eta(\cdot,\cdot, v(\cdot,\cdot))$ and
\begin{eqnarray*}
\left\{
\begin{array}{l}
a_u=D_y\eta(\tau,\xi,v(\tau,\xi))a_v\\\\
p_u=D_y\eta(\tau,\xi,v(\tau,\xi))p_v+D_x\eta(\tau,\xi,v(\tau,\xi))\\\\
X_u=D_y\eta(\tau,\xi,v(\tau,\xi))X_v+2D_{xy}\eta(\tau,\xi,v(\tau,\xi))p^*_v\\~~~~~~~+D_{xx}\eta(\tau,\xi,v(\tau,\xi))
+D_{yy}\eta(\tau,\xi,v(\tau,\xi))p_vp_v^*.
\end{array}\right.
\end{eqnarray*}
\end{proposition}
However, contrary to classical SPDE, the resulting PDIE from MSPDIE
(\ref{i11}) due to Doss-Sussman transformation is not necessarily
MPDIE studied by N'zi and Ouknine (see \cite{Mn2}). Therefore, we
need the following version of viscosity solution for resulting PDIE
obtain by Doss-Sussman transformation.
\begin{corollary}\label{corollary4.8}
Assume that the assumptions {\rm({H1})--({H7})} hold. Let us define
and consider $(\tau,\xi)\in \mathcal{M}_{0,T}^B\times
L^2\left(\mathcal{F}^{B}_{\tau}; \R^d\right)$, $u\in
\mathcal{C}\left(\mathbf{F}^B, [0,T]\times\R^d\right)$.
\begin{description}
\item[\rm(1)] for $(a_u,X_u,p_u)\in\mathcal{J}^{1,2,+}_{g}u(\tau,\xi)$,
$u$ satisfies (\ref{E:def1}) if and only if $v(\cdot,\cdot) = \varepsilon(\cdot,\cdot,
u(\cdot,\cdot))$ satisfies
\begin{eqnarray}
V_{\widetilde{f}}(\tau,\xi,a_v,p_v,X_v)+\frac{\varphi'_l
(\eta(\tau,\xi,v(\tau,\xi))}{D_y\eta(\tau,\xi,v(\tau,\xi))}\leq
0;\label{E:def2}
\end{eqnarray}

\item[\rm(2)] for $(a_u,X_u,p_u)\in\mathcal{J}^{1,2,-}_{g} u(\tau,\xi)$,
$u$ satisfies (\ref{E:def01}) if and only if $v(\cdot,\cdot) = \varepsilon(\cdot,\cdot,
u(\cdot,\cdot))$ satisfies
\begin{eqnarray}\label{E:def02}
V_{\widetilde{f}}(\tau,\xi,a_v,p_v,X_v)+\frac{\varphi'_l
(\eta(\tau,\xi,v(\tau,\xi))}{D_y\eta(\tau,\xi,v(\tau,\xi))}\geq
0;
\end{eqnarray}
where $(a_v, p_v,X_v)$ is defined by Proposition \ref{Propo} and
\begin{eqnarray*}
\widetilde{f}(t,x,y,(\theta^k)_{k\geq 1})&=&\frac{1}{D_y\eta(t,x,y)}\left[f\left(t,x,\eta(t,x,y),D_y\eta(t,x,y)\theta^{k}+\eta^1_k(t,x,y)\right)\right.\\
&&\left.-\frac{1}{2}(g\partial_u g)(t,x,\eta(t,x,y))+L_x\eta(t,x,y)+\lambda\langle\sigma^*(x)D_{xy}\eta(t,x,y),\sigma(x)p_v\rangle\right.\\
&&\left.+\frac{1}{2}\lambda D_{yy}\eta(t,x,y)|\sigma(x)p_v|^2\right].
\end{eqnarray*}
with $\theta^k=\int_{\R}\langle p_v,\sigma(x)u\rangle p_k(u)\nu(du)$ and $\lambda=1+\int_{\R}u^2\nu(du)$.
\end{description}
\end{corollary}
\begin{proof}
Let $(\tau,\xi)\in \mathcal{M}_{0,T}^B\times
L^2\left(\mathcal{F}^{B}_{\tau}; \R^d\right)$ be given and $(a_u,
p_u,X_u)\in\mathcal{J}^{1,2,+}_{g}u(\tau,\xi)$. We assume that $u$
is a stochastic subsolution of MSPDIE (\ref{i11}), i.e.
\begin{eqnarray*}
u(\tau,\xi)&\in & {\rm Dom}(\varphi),\;\;\;\;\;\; \forall\;
(\tau,\xi)\in\mathcal{M}_{0,T}^B\times
L^2\left(\mathcal{F}^{B}_{\tau};\R^d\right),\;\;\; \P\mbox{-a.s.},
\end{eqnarray*}
such that
\begin{eqnarray*}
V_{f}(\tau,\xi,a,p,X)+\varphi'_l(u(\tau,\xi)-\frac{1}{2}(g\partial_u\
g)(\tau,\xi,u(\tau,\xi))\leq 0,\,\; \P\mbox{-a.s.}
\end{eqnarray*}
In view of Proposition \ref{Propo} and since $D_y\eta(t,x,y)>0,$ for
all $(t,x,y)$, we obtain by little calculation
\begin{eqnarray*}
V_{\widetilde{f}}(\tau,\xi,a_v,p_v,X_v)+\frac{\varphi'_{l}(\eta(\tau,\xi,v(\tau,\xi)))}{D_y\eta(\tau,\xi,v(\tau,\xi))}\leq
0.
\end{eqnarray*}
The
converse part of (1) can be proved similarly. In the same manner one
can show the second assertion (2).
\end{proof}
\section{Probabilistic representation result for stochastic viscosity solution to MSPDIEs}
In this section, we aim to show that the solution of multivalued
BDSDE with jump gives the viscosity solution of a semi-linear MSPDIE
in the Markovian case.
\subsection{A class of reflected diffusion process}
We now introduce a class of diffusion process. Let
$\sigma:\R^d\rightarrow\R^{d}$ be a uniformly bounded
function satisfying the uniform Lipschitz condition with some
constant $C > 0$, for all $x, y\in\R^d$:
\begin{eqnarray}
|\sigma(x)-\sigma(y)|\leq C|x-y|.
\end{eqnarray}
For each $(t,x)\in[0,T]\times \R^d$, from \cite{Slo} and reference therein, let $\{X^{t,x}_s, s\in[t,T]\}$
be a unique pair of progressively measurable process, which is a
solution to the following stochastic differential equation:
\begin{eqnarray}
X_s^{t,x} = x+\int^{s}_t \sigma(X_{r^-}^{t,x})\,{\rm
d}L_r.\label{rSDE}
\end{eqnarray}
Furthermore, we have the following proposition.
\begin{proposition}\label{P:continuity00}
There exists a constant $ C>0 $ such that for all $0\leq t<t'\leq T$
and $x,\,x'\in\R^d$, such that
$$\mathbb{E}\left[\sup_{0\leq s\leq
T}\left|X^{t,x}_{s}-X^{t',x'}_{s}\right|^4\right] \leq C\left(
|t'-t|^2+|x-x'|^4\right)$$
\end{proposition}
\subsection{Existence of viscosity solution for MSPDIEs}
Fix $T > 0$ and for all $(t, x)\in [0,T]\times\R^d$, let
$X^{t,x}_s,\; s\in [t,T]$ denote the solution of the SDE
(\ref{rSDE}). And we suppose now that the data $(\xi,f , g)$ of the
multi-valued BDSDE with jump take the form
\begin{eqnarray*}
\xi&=&u_0(X^{t,x}_T),\\
f(s,y,z)&=&f(s,X_s^{t,x},y,z),\\
g(s,y)&=&f(s,X_s^{t,x},y).
\end{eqnarray*}
And we give the following assumptions:\newline We assume that
$u_0\in C(\R^d;\R),\, f\in
C([0,T]\times\R^d\times\R\times\ell^2;\R)$ and $g\in
C([0,T]\times\R^d\times\R;\R)$ such that assumptions {\rm
(H1)--(H7)} hold. It follows from the results of the Section 3 that,
for all $(t, x)\in[0, T]\times\R^d$, there exists a unique triplet
$(Y^{t,x},Z^{t,x}, U^{t,x})$ for the solution of the following
\newline $
\begin{array}{ll}
\displaystyle (1)\ \ (Y^{t,x}_s,U^{t,x}_s)\in \partial \varphi, \
\,{\rm d}\mathbb{P}\otimes
\,{\rm d}s,\mbox{-a.e. on}\ [t,T]\\
\displaystyle (2)\ \ Y^{t,x}_s+\int_s^TU^{t,x}_r\,{\rm
d}r=u_0(X^{t,x}_T)+\int_s^Tf(r,X^{t,x}_r,Y^{t,x}_r,Z^{t,x}_r)\,{\rm
d}r+\int_s^Tg(r,X^{t,x}_r,Y^{t,x}_r)\,{\rm
d}B_r\\\qquad\qquad\qquad\qquad\qquad
\displaystyle-\sum_{i=1}^{\infty}\int_s^T(Z^{t,x})^{(i)}_r\,{\rm
d}H^{(i)}_r,\ t\leq s \leq T.\label{eqmarkov1}
\end{array}
$
\\
We extend processes $Y^{t,x},\, Z^{t,x},\,U^{t,x}$ on $[0,T]$ by
putting $ Y^{t,x}_s=Y^{t,x}_t,\, Z^{t,x}_s=0,\,\,U^{t,x}_s=0,\;\;
s\in [0,t]$.

We have this result whose proof is similar to that of Theorem 2.1 appear in \cite{PP}
\begin{proposition}\label{Prop}
Let the ordered triplet $(Y^{t,x}_s, U^{t,x}_s,
Z^{t,x}_s)$ be the unique solution of the multi-valued BDSDE
\eqref{eqmarkov1}. Then, for $(s, t, x)\in[0, T ]\times[0, T
]\times\R^d$, the random field $(s, t, x)\mapsto
\E'(Y^{t,x}_s)$ is a.s. continuous ($Y^{t,x}$ has jump), where $\E'$
is the expectation with respect to $\P'$, introduced at page 3.
\end{proposition}
We are ready now to derive our main result in this section.
\begin{theorem}
Assume the assumptions {\rm(H1)--(H7)} be satisfied. Then, the
function $u(t, x)$ defined by $u(t, x)=Y^{t,x}_t$ is a stochastic
viscosity solution of MSPDIE \eqref{i11}
\end{theorem}
\begin{proof}
The fact that $u(t,x)=Y^{t,x}_t=\E'(Y^{t,x}_t)$, does not depend on $\omega'$, it follows from Proposition \ref{Prop} that $u\in C(\mathcal{F}^{B},[0,T]\times \R^d)$. Next, for all $(\tau,\xi)\in\mathcal{M}^{B}(0,T)\times
L^{2}(\mathcal{F}^{B},\R^d)$,
\begin{eqnarray*}
\varphi(u(\tau(\omega),\xi(\omega)))=
\varphi\left(Y^{\tau(\omega),\xi(\omega)}_{\tau(\omega)}\right)<\infty,
\  \P\mbox{-a.s.},
\end{eqnarray*}
which implies that $u(\tau,\xi)\in {\rm Dom}(\varphi)\; \P$-a.s.Thus it remains to show that $u$ is the stochastic viscosity solution to MSPDIE
\eqref{i11}. In other word, using Corollary \ref{corollary4.8}, it suffices to prove that $v(t, x) =\varepsilon(t, x, u(t, x))$ satisfies \eqref{E:def2} and
\eqref{E:def02}. In this fact, for each $(t,x)\in[0,T]\times\overline{\Theta},\; \delta>0$, let $\{(Y^{t,x,\delta}_{s},Z^{t,x,\delta}_{s}),\,\ 0\leq s\leq T\}$ denote the solution of the following BDSDE:
\begin{eqnarray}
Y^{t,x,\delta}_s+\frac{1}{\delta}\int_s^TD\varphi_{\delta}(Y^{t,x,\delta}_r)\,{\rm
d}r
&=&u_0(X^{t,x}_T)+\int_s^Tf(r,X^{t,x}_r,Y^{t,x,\delta}_r,Z^{t,x,\delta}_r)\,{\rm
d}r\nonumber\\
&&+\int_s^T g(r,X^{t,x}_r,Y^{t,x,\delta}_r) \,{\rm
d}B_r\nonumber\\
&&-\sum_{i=1}^{\infty}\int_s^T(Z^{t,x,\delta})^{(i)}_r\,{\rm
d}H^{(i)}_r.\label{eqmarkovapp}
\end{eqnarray}
Setting $Y^{t,x,\delta}_t=u^{\delta}(t,x)$, it is shown by Theorem 3.6 in \cite{MY}, that the function $v^{\delta}(t, x) =
\varepsilon(t, x, u^{\delta}(t, x))$ is an $\omega$-wise viscosity
solution to this MSPDIE:
\begin{eqnarray}
\left\{
\begin{array}{l}
(i)\;\displaystyle\left(\frac{\partial v^{\delta}}{\partial
t}(t,x)-\left[
\mathcal{L}v^{\delta}(t,x)+\widetilde{f}_{\delta}(t,x,v^{\delta}(t,x),\sigma^{*}(x)\nabla
v^{\delta}(t,x))\right]\right)=0,\,\,\
(t,x)\in[0,T]\times\R^d,\\\\
(ii)\; v(T,x)=u_{0}(x),\,\,\,\,\,\,\ x\in\R^d,
\end{array}\right.\label{SPDE}
\end{eqnarray}
where
$$\widetilde{f}_{\delta}(t,x,y,z)=\widetilde{f}(t,x,y,z)-\frac{\frac {1}{\delta} D\varphi_{\delta}(\eta(t,x,y))}{D_y\eta(t,x,y)}.$$
Moreover, letting Lemma 3.5 provide, along a subsequence that
\begin{eqnarray}
|v^{\delta}(t,x)-v(t,x)|\rightarrow 0, \ \mbox{a.s.,\ as}\
\delta\rightarrow 0, \label{CV}
\end{eqnarray}
for all $(t,x)\in[0,T]\times\R^d$.

On the other hand, for all $(\tau,\xi)\in\mathcal{M}^{B}(0,T)\times
L^{2}(\mathcal{F}^{B},\R^d)$ and $\omega\in\Omega$ be fixed, let consider
$(a_v,p_v,X_v)\in\mathcal{J}^{1,2,+}_{0}(v(\tau(\omega),\xi(\omega)))$.
Thus, since $v^{\delta}$ is an $\omega$-wise viscosity
solution to the MSPDIE \eqref{SPDE}, and by Crandall- Ishii-Lions in \cite{CIL}, there
exist sequences
\begin{eqnarray*}
\left\{
\begin{array}{ll}
\delta_n(\omega)\searrow 0,\\\\
(\tau_n(\omega),\xi_n(\omega))\in[0,T]\times\R^d,\\\\
(a_v^{n},p_v^{n},X_v^{n})\in\mathcal{J}^{1,2,+}_{0}(v^{\delta_n}(\tau_n(\omega),\xi_n(\omega)))
\end{array}
\right.
\end{eqnarray*}
satisfying
\begin{eqnarray*}
(\tau_n(\omega),\xi_n(\omega),a_v^{n},p_v^{n},X_v^{n},v^{\delta_n}(\tau_n(\omega),
\xi_n(\omega)))&\to&
(\tau(\omega),\xi(\omega),a_v,p_v,X_v,v(\tau(\omega),\xi(\omega)))\\
n&\to& \infty,
\end{eqnarray*}
such that for $(\tau_n(\omega),\xi_n(\omega))\in [0,T]\times\R^d$,
\begin{eqnarray}
V_{\widetilde{f}(\omega)}(\tau_n(\omega),\xi_n(\omega),a^n_v,X^n_v,p^n_v)+\frac{\frac {1}{\delta_n} D\varphi_{\delta_n}(\eta(\tau_n(\omega),\xi_n(\omega),v^{\delta_n}
(\tau_n(\omega),\xi_n(\omega))))}
{D_y\eta(\tau_n(\omega),\xi_n(\omega),v^{\delta_n}(\tau_n(\omega),\xi_n(\omega)))}
\leq 0.\label{V01}
\end{eqnarray}
In order to simplify the notation, we remove the
dependence of $\omega$. Let $y\in {\rm Dom}(\varphi)$ such that
$y< u(\tau,\xi)=\eta(\tau,\xi,v(\tau,\xi))$. The ucp convergence
of $v^{\delta_n}$ to $v$ implies that there exists $n_0 > 0$ such
that  $\forall\, n\geq n_0,\;\; y<\eta(\tau_n,\xi_n,v^{\delta_n}(\tau_n,\xi_n))$. Therefore, inequality (\ref{V01}) yields
\begin{eqnarray*}
&&\left(\eta(\tau_n,\xi_n,v^{\delta_n}(\tau_n,\xi_n))-y\right)
V_{\widetilde{f}_{\delta_n}}\left(\tau_n,\xi_n,a^n_v,X^n_v,p^n_v\right)\nonumber\\
&\leq&
\left[\varphi(y)-\varphi(J_{\delta_n}(\eta(\tau,\xi,v^{\delta_n}(\tau,\xi))))\right]\frac{1}
{D_y\eta(\tau_n,\xi_n,v^{\delta_n}(\tau_n,\xi_n))}.\label{V1}
\end{eqnarray*}
Taking the limit in this last inequality, we get for all
$y<\eta(\tau,\xi,v(\tau,\xi))$
\begin{eqnarray*}
V_{\widetilde{f}}(\tau,\xi,a_v,X_v,p_v)\leq
-\frac{\varphi(\eta(\tau,\xi,v(\tau,\xi)))-\varphi(y)}{\eta(\tau,\xi,v(\tau,\xi))-y}\frac{1}
{D_y\eta(\tau,\xi,v(\tau,\xi))},
\end{eqnarray*}
which implies that
\begin{eqnarray*}
&&V_{\widetilde{f}}(\tau,\xi,a_v,X_v,p_v)+\frac{\varphi'_l(\eta(\tau,\xi,v(\tau,\xi)))}{D_y\eta(\tau,\xi,v(\tau,\xi))}\leq
0,\label{V1}
\end{eqnarray*}
and derives that $v$ satisfies (\ref{E:def2}). Hence, according to Corollary
\ref{corollary4.8}, $u$ is a stochastic viscosity
subsolution of MSPDIE \eqref{i11}. By similar arguments, one
can prove that $u$ is a stochastic viscosity supersolution of MSPDIE \eqref{i11}. This completes the proof.
\end{proof}

%\begin{acknowledgements}
%\end{acknowledgements}



\end{document}